\documentclass[14pt]{amsart}
\usepackage[cp1251]{inputenc}
\usepackage[english,russian]{babel}
\usepackage{amsmath}
\usepackage{amssymb}
\usepackage{amsfonts}
\usepackage{graphicx}

\newtheorem{lemma}{Lemma}
\newtheorem{theorem}{Theorem}
\newtheorem{definition}{Definition}
\newtheorem{corollary}{Corollary}
\newtheorem{proposition}{Proposition}

\begin{document}
\renewcommand{\refname}{References}
\renewcommand{\proofname}{Proof.}
\renewcommand{\figurename}{Fig.}

\thispagestyle{empty}

\begin{center}
\textbf{\sc The inter-relationship between isomorphisms of commutative and isomorphisms of non-commutative $\log$-algebras}
\end{center}

\vspace{10pt}

\begin{center}
\textbf{\sc R.Abdullaev, A.Azizov}
\end{center}

\vspace{10pt}

\author[R.Abdullaev]{{\bf R.Abdullaev}\protect}

\author[A.Azizov]{{\bf A.Azizov}\protect}

\address{Rustam Zoirovich Abdullaev  
\newline\hphantom{iii} Tashkent university of information technologies,
\newline\hphantom{iii} pr. A.Temur, 108,
\newline\hphantom{iii} 100200, Tashkent, Uzbekistan}%
\email{arustambay@yandex.com}%

\address{Azizkhon Nodirovich Azizov  
\newline\hphantom{iii} National university of Uzbekistan,
\newline\hphantom{iii} st. Universitet, 4,
\newline\hphantom{iii} 100029, Tashkent, Uzbekistan}%
\email{azizov.07@mail.ru}%

\sloppy
\noindent{\bf{Abstract.}} This paper establishes a necessary and sufficient condition for the coincidence of non-commutative $\log$-algebras constructed from different exact normal semifinite traces. Consequently, we provide a criterion for the isomorphism of $\log$-algebras built on non-commutative von Neumann algebras with different exact normal semifinite traces. Additionally, we demonstrate a connection between the isomorphism of non-commutative $\log$-algebras and the isomorphism of the corresponding $\log$-algebras constructed on the center of these von Neumann algebras. Furthermore, we present a necessary and sufficient condition for the isomorphism of $\log$-algebras derived from different von Neumann algebras of type $I_n$.
\medskip


\section{Introduction}
\sloppy
In the work \cite{1}, commutative and noncommutative log-algebras (logarithmically integrable elements associated with von Neumann algebras) were introduced and their properties were studied. In particular, it was shown that such an algebra is a complete metric space, with algebraic operations being continuous for the set of variables. The works \cite{2}, \cite{3} consider commutative $log$-algebras constructed with finite and $\sigma$-finite measures, respectively, and study their algebraic and topological properties.

This paper presents a necessary and sufficient condition for the coincidence of non-commutative $\log $-algebras constructed from different exact normal semifinite traces. Consequently, we prove a criterion for the isomorphism of $\log $-algebras constructed on non-commutative von Neumann algebras from different exact normal semifinite traces. We also establish a connection between the isomorphism of non-commutative $\log $-algebras and the isomorphism of the corresponding $\log $-algebras constructed on the center of this von Neumann algebra. In addition, we establish a necessary and sufficient condition for the isomorphism of $\log $-algebras constructed from different von Neumann algebras of type $I_n$. In particular, the isomorphism is reduced to the boundedness of the corresponding numerical sequences.

These $\log$-algebras are interesting because they are close to the class of Nevanlinna functions holomorphic in the disk and integrable with the logarithm on the disk's boundary of functions \cite{12}. Since the logarithm of a holomorphic function is a subharmonic function, the Nevanlinna class establishes a connection between holomorphic functions and potential theory. Besides, the operators integrable with respect to states of any degree $p\in [1,+\infty)$ were also considered in \cite{13}.

Let us consider non-commutative algebras constructed on von Neumann algebras. All necessary definitions, notations, and results of the theory of von Neumann algebra are taken from \cite{4}, and the theory of non-commutative integration from \cite{5}. Let $M$ be a von Neumann algebra with exact normal semi-finite trace $\tau $. By $L_{0} \left(M,\tau \right)$ we denote the measurable operators associated with $M$. Consider the space
\[L_{\log } \left(M,\tau \right)=\left\{T\in L_{0} \left(M,\tau \right):\tau \left(\log \left(1+\left|T\right|\right)\right)<\infty \right\}\]
in the sense of \cite{1}. $L_{\log } \left(M,\tau \right)$ is an \textit{$F$}-space with respect to the $F$-norm
\[\left\| T\right\| _{\log } =\int _{0}^{1}\log \left(1+\mu _{x} (T)\right)dx \]
in the case of finite $\tau $ and
\[\left\| T\right\| _{\log } =\int _{0}^{\infty }\log \left(1+\mu _{x} (T)\right)dx \]
in the case of semi-finite $\tau $, where $\mu _{x} (T)$ is a permutation of the operator $T$ \cite{5}.

Let $\mu $ and $\nu $ be exact normal semifinite traces on a von Neumann algebra $M$, and let $h=\frac{d\nu }{d\mu } $ denote the Radon-Nikodym derivative of $\nu $ with respect to $\mu $, i.e., a central positive operator adjoint to $M$ such that $\nu (x)=\mu (hx)$ holds for all $x\in M$. Moreover, $h^{-1} $ exists and $h^{-1} \in Z(L_0(M,\nu))$ \cite[Theorem 1, (ii) on p. 63, Corollary, (ii) on p. 65]{14}, where $Z(A)$ or $Z_A$ denotes the center of $A$. Now from the equality $\nu (x)=\mu (hx)$ we obtain $\nu (h^{-1} x)=\mu (h^{-1} hx)=\mu (x)$, i.e. $h^{-1} $ is the Radon-Nikodym derivative of the trace $\mu $ with respect to $\nu $.

\begin{lemma}[\cite{1}] Let $S,T\in S(M,\mu )$. Then

a) From $\left\| T\right\| _{\log } >0$, follows $T\ne 0$;

b) $\left\| \alpha T\right\| _{\log } \le \left\| T\right\| _{\log } $ for all numbers $\alpha $, $\left|\alpha \right|\le 1$;

c) If $T\in L_{\log } \left(M,\mu \right)$, then ${\mathop{\lim }\limits_{\alpha \to 0}} \left\| \alpha T\right\| _{\log } =0$;

d) $\left\| S+T\right\| _{\log } \le \left\| S\right\| _{\log } +\left\| T\right\| _{\log } $;

e) $\left\| S\cdot T\right\| _{\log } \le \left\| S\right\| _{\log } + \left\| T\right\| _{\log } $.
\end{lemma}

From properties a), b), c), d), it follows that the function $\left\| \cdot \right\| _{\log } $ is an $F$-norm on the $F$-space $L_{\log } \left(M,\mu \right)$, and from property e) we obtain that the $F$-space $L_{\log } \left(M,\mu \right)$ is a topological algebra with respect to the topology generated by the metric $\rho (S,T)=\left\| S-T\right\| _{\log } $.

\section{Criterion for isomorphism of non-commutative $log$-algebras}

Let $M$ be a von Neumann algebra with exact normal semifinite traces $\mu $ and $\nu$, $h=\frac{d\nu }{d\mu } $ the Radon-Nikodym derivative.
\begin{theorem}\label{th1} $L_{\log } \left(M,\mu \right)\subset L_{\log } \left(M,\nu \right)$ if and only if $h\in M$.
\end{theorem}
\begin{proof} Let $h\in M$ and $f\in L_{\log } \left(M,\mu \right)$, i.e.
\begin{equation}\label{1'}
\int_{\Omega }\log \left(1+\left|f(z)\right|\right)d\mu <\infty.
\end{equation} Since $h$ is a central element, without loss of generality, we can assume that we are considering the commutative case, i.e., functional $\log $-algebras. Therefore, we can assume that $L_{\log } \left(M,\mu \right)=L_{\log } \left(\Omega ,\mu \right)$. Then, by virtue of the Radon-Nikodym theorem and \eqref{1'} we have $$\int _{\Omega }\log \left(1+\left|f(z)\right|\right)d\nu =\int _{\Omega }\left(h\log \left(1+\left|f(z)\right|\right)\right)d\mu \le $$ $$\leq \left\| h\right\| _{\infty } \int _{\Omega }\log \left(1+\left|f(z)\right|\right)d\mu <\infty.$$ Hence $f\in L_{\log } \left(M,\nu \right)$, i.e. $L_{\log } \left(M,\mu \right)\subset L_{\log } \left(M,\nu \right)$.

Conversely, let $0<h\in Z(L_{0} \left(M,\mu \right))\setminus M$. As was said above, we can assume that $L_{1} \left(M,\mu \right)=L_{1} \left(\Omega ,\mu \right)$ and $M=L_{\infty } \left(\Omega ,\mu \right)$, i.e. $M$ is a function space. Then for some projection $p$ $(\mu(p)<\infty)$ we have $0<hp\in L_{1} \left(M,\mu \right)\setminus M$. Now we can construct an infinite sequence of sets $$M_{n} =\left\{z\in \Omega ': n\le h(z)p\le n+1\right\}$$ where $\Omega '\subset \Omega $, $\mu (\Omega ')<\infty $. This set is infinite because $hp$ is unbounded, and for any $n$, we have $M_n\neq \emptyset$. Now we consider an infinite subset of natural numbers $N_{0} =\left\{n\in N:\mu (M_{n} )>0\right\}$. We rename the elements of the set $N_{0} $ as follows $N_{0} =\left\{n_{1} ,n_{2} ,\ldots \right\}$, $n_{k} <n_{k+1} $.

Consider the function
\[g(z)=\left\{\frac{1}{k^{2} \mu (M_{n_{k} } )} ;z\in M_{n_{k} } \right\}.\]
For $z\in \Omega \backslash \cup _{k} M_{n_{k} } $ we have $g(z)=0$. Let's put $f(z)=e^{g(z)} -1$, then \begin{equation} \label{eq1} \int _{\Omega }\log \left(1+\left|f(z)\right|\right)d\mu =\sum _{k=1}^{\infty }\frac{\mu \left(M_{n_{k} } \right)}{k^{ 2} \mu \left(M_{n_{k} } \right)} =\sum _{k=1}^{\infty }\frac{1}{k^{2} } <\infty.
\end{equation} However \[\int_{\Omega } \log (1+|f(z)|)d\nu =\nu (\log (1+|f(z)|))=\mu (h(z)\log (1+f(z)))\]
\begin{equation} \label{eq2}
=\mu \left(hg\right)\ge \sum _{k=1}^{\ infty }\frac{n_{k} \cdot \mu \left(M_{n_{k} } \right)}{k^{2} \mu \left(M_{n_{k} } \right)} =\sum _{k=1}^{\infty }\frac{n_{k} }{k^{2} } \ge \sum _{k=1}^{\infty }\frac{1}{k} =\infty .
\end{equation}
From \eqref{eq1} and \eqref{eq2} it follows that $f\in L_{\log } \left(M,\mu \right)$ and $f\notin L_{\log } \left(M,\nu \right)$, i.e. $L_{\log } \left(M,\mu \right)$ is not a subset of $L_{\log } \left(M,\nu \right)$, for $h\in L_{0} \left(M,\mu \right)\backslash M$.

So from $L_{\log } \left(M,\mu \right)\subset L_{\log } \left(M,\nu \right)$ it follows $h\in M$.
\end{proof}

From theorem \ref{th1} we obtain

\begin{corollary}\label{cor1} $L_{\log } \left(M,\mu \right)=L_{\log } \left(M,\nu \right)$ if and only if $h,h^{-1} \in M$.
\end{corollary}

Let $M$ and $N$ be non-commutative von Neumann algebras with faithful normal semifinite traces $\mu $ and $\nu $, respectively.

\begin{definition}\label{def1}
We call traces $\mu $ and $\nu $ \textit{$\alpha$}-equivalent if there exists a $*$-isomorphism $\alpha :M\to N$ such that one of the following equivalent conditions holds:

$(i).$ $L_{\log } \left(N,\nu \right)=L_{\log } \left(N,\mu \circ \alpha ^{-1} \right);$

$(ii).$ $\frac{d\nu }{d\mu \circ \alpha ^{-1} } ,\frac{d\mu \circ \alpha ^{-1} }{d\nu } \in N.$
\end{definition}
The functional $\mu \circ \alpha ^{-1} $ is a faithful normal semifinite trace on the von Neumann algebra $N$. It follows from Corollary \ref{cor1} that $(i)$ and $(ii)$ are equivalent.

\begin{theorem}\label{th2} The algebras $L_{\log } \left(M,\mu \right)$ and $L_{\log } \left(N,\nu \right)$ are isomorphic if and only if $\mu $ and $\nu $ are \textit{$\alpha$}-equivalent.
\end{theorem}
\begin{proof} Let $\mu $ and $\nu $ be \textit{$\alpha$}-equivalent, i.e. there exists an isomorphism $\alpha :M\to N$ such that condition $(i)$ from \mbox{definition \ref{def1}} holds
\begin{equation} \label{eq3}
L_{\log } \left(N,\mu \circ \alpha ^{-1} \right)=L_{\log } \left(N,\nu \right).
\end{equation}

From \cite[Proposition 3]{2} we have $\alpha(\log(1+|f|))=\log(1+\alpha(|f|))$ for all $f \in L_0(M)$ and for commutative von Neumann algebras $M$. Since this equality involves only one element, then this equality is also true for non-commutative von Neumann algebras $M$.

\sloppy
It follows that $\alpha(f) \in L_{\log}(N,\mu \circ \alpha ^{-1})$ for all $f \in L_{\log}(M,\mu)$. Similarly, using the *-isomorphism $\alpha^{-1}$ we obtain $\alpha^{-1}(h) \in L_{\log}(M,\mu)$ for all $h \in L_{\log}(N,\mu \circ \alpha ^{-1})$. Therefore
\begin{equation} \label{eq4}
\alpha(L_{\log}(M,\mu))=L_{\log}(N,\mu \circ \alpha ^{-1}).
\end{equation}

From \eqref{eq3} and \eqref{eq4} we obtain $$\alpha (L_{\log } \left(M,\mu \right))=L_{\log } \left(N,\mu \circ \ alpha ^{-1} \right)=L_{\log } \left(N,\nu \right),$$ i.e. $L_{\log } \left(M,\mu \right)$ and $L_{\log } \left(N,\nu \right)$ are isomorphic.

Conversely, let $\alpha '$ be an *-isomorphism from $L_{\log } \left(M,\mu \right)$ to $L_{\log } \left(N,\nu \right)$, then
\begin{equation} \label{eq5}
\alpha '\left(L_{\log } \left(M,\mu \right)\right)=L_{\log } \left(N,\nu \right) .
\end{equation}

Then $\alpha '$ maps bounded elements from $L_{\log } \left(M,\mu \right)$ to bounded elements from $L_{\log } \left(N,\nu \right)$, i.e. the restriction of $\alpha '$ to $M$ is a *-isomorphism from $M$ to $N$. Moreover, the *-isomorphism from $M$ to $N$ satisfies the condition of \textit{$\alpha$}-equivalence of the traces of $\mu $ and $\nu .$ From \cite[Proposition 3]{6} we have
\begin {equation} \label{eq6}
L_{\log } \left(N,\mu \circ \alpha ^{-1} \right)=\alpha '\left(L_{\log } \left(M,\mu \right)\right)
\end{equation}
From \eqref{eq5} and \eqref{eq6} we obtain
\[L_{\log } \left(N,\mu \circ \alpha ^{-1} \right )=L_{\log } \left(N,\nu \right),\]
i.e. $\mu $ and $\nu $ are $\alpha$-equivalent.
\end{proof}

\section{Isomorphisms of non-commutative $\log$-algebras that can be reduced to isomorphisms of the corresponding commutative $\log$-algebras}

Let $\mu $ and $\nu $ be exact normal semifinite traces on a von Neumann algebra $M$. If $\alpha $ is some automorphism of $M$, then the functional $\nu \circ \alpha ^{-1} $ is also an exact normal semifinite trace. Denote by $d\mu /d\nu \circ \alpha ^{-1} $ the Radon--Nikodym derivative of the trace $\mu $ with respect to the trace $\nu \circ \alpha ^{-1} $.

The set of $*$ - automorphisms $\alpha $ of the von Neumann algebra $M$ satisfying the conditions
$\frac{d\mu }{d\nu \circ \alpha ^{-1}}, \ \frac{d\nu \circ \alpha ^{-1}}{d\mu}$ are bounded, we denote by $Aut(M\left(\mu ;\nu \right))$.

Now we can reformulate Theorem \ref{th2} (for the case $M=N$), establishing a criterion for $*$-isomorphism of non-commutative $\log$-algebras.

\begin{theorem}\label{th3}
Let $\mu $ and $\nu $ be faithful normal semifinite traces on a von Neumann algebra $M$. The following conditions are equivalent:

$(i)$ log algebras $L_{\log } \left(M;\mu \right)$ and $L_{\log } \left(M;\nu \right)$ are *-isomorphic;

$(ii)$ $Aut\, M\left(\mu ;\nu \right)\ne \emptyset $.
\end{theorem}

Let $\mu $ and $\nu $ be faithful normal semifinite traces on a von Neumann algebra $M$. Denote by $Z_{M} $ the center of the von Neumann algebra $M$, and by $\mu '$ and $\nu '$ the restrictions of the traces $\mu $ and $\nu $ to the center of $Z_{M} $, respectively. Since the automorphism $\alpha $ of the von Neumann algebra $M$ restricts to an automorphism on $Z_{M} $, the equality
\[\frac{d\mu }{d\nu \circ \alpha ^{-1} } =\frac{d\mu '}{d\nu '\circ \alpha ^{-1} } \]
implies the inclusion
\begin{equation}\label{eq8}
\left\{\alpha\left|_{Z_{M}} :\right. \alpha \in Aut(M\left(\mu ;\nu \right))\right\}\subset Aut(Z_{M} \left(\mu ';\nu '\right)).
\end{equation}

Suppose that
\[\alpha (L_{\log } \left(M;\mu \right))=L_{\log } \left(M;\nu \right)\]
i.e. $Aut(M\left(\mu ;\nu \right))\ne \emptyset $. It follows that
\[\left\{\alpha \left|_{Z_{M}} :\right. \alpha \in Aut(M\left(\mu ;\nu \right))\right\}\ne \emptyset \]
Therefore, from inclusion \eqref{eq8} we obtain that $Aut(Z_{M}\left(\mu '; \nu '\right))\ne \emptyset$. From these considerations, we obtain the following assertion.

\begin{proposition}
If log-algebras $L_{\log } \left(M;\mu \right)$ and $L_{\log } \left(M;\nu \right)$ are *-isomorphic, then log-algebras $L_{\log } \left(Z_{M} ;\mu '\right)$ and $L_{\log } \left(Z_{M} ;\nu '\right)$ are also *-isomorphic. The converse is not always true.
\end{proposition}

\textbf{Example 1.} Let
\[M=Z_{1} \otimes M_{1} \otimes Z_{2} \otimes M_{2} \]
where $Z_{1} $ and $Z_{2} $ are *-isomorphic commutative von Neumann algebras, and $M_{1} $ and $M_{2} $ are factors of type $I_{n_1}$ and $I_{n_2}$, respectively. Denote by $M^{I} =Z_{1} \otimes M_{1} $ and by $M^{II} =Z_{2} \otimes M_{2} $. Then it follows from \cite[Corollary 2.3.4]{10} that any automorphism $\beta: M \to M$ maps $M^{I}$ to $M^{I}$ and maps $M^{II}$ to $M^{II}$.

\sloppy
On the von Neumann algebra $M^{I} $ we consider faithful normal semifinite traces $\mu _{1} $ and $\nu _{1} $ such that $Aut\, M^{I} \left(\mu _{1} ;\nu _{1} \right)\ne \emptyset $. We denote the *-isomorphism of the algebras $Z_{1} $ and $Z_{2} $ by $\pi $ and the restrictions of the traces $\mu _{1} $ and $\nu _{1} $ to $Z_{1} $ by $\mu '_{1} $ and $\nu '_{1} $, respectively. Then on $Z_{2} $ one can define exact normal semi-finite traces $\mu '_{2} =\nu '_{1} \circ \pi ^{-1} $ and $\nu '_{2} =\mu '_{1} \circ \pi ^{-1} $.

Let $Z_{M} =Z_{1} \oplus Z_{2} $ and $\mu '=\mu '_{1} \oplus \mu '_{2}, \ \nu '=\nu '_{1} \oplus \nu '_{2} $ -- exact normal semifinite traces on $Z_{M} $. Let $x_1\in Z_1$ and $x_2 \in Z_2$, i.e. $(x_1,x_2)\in (Z_1,Z_2)=Z_M$. Consider a map $\alpha$ on $Z_{M}$ such that $\alpha ^{-1} \left(x_{1} ;x_{2} \right)=\alpha \left(x_{1} ; x_{2} \right)=\left(\pi ^{-1} x_{2} ;\pi x_{1} \right)$. Obviously, $\alpha$ will be an automorphism on $Z_M$. Then from the equalities \[\begin{array}{l} {\nu '\circ \alpha ^{-1} \left(x_{1} \oplus x_{2} \right)=\nu '\left(\ pi ^{-1} \left(x_{2} \right)\oplus \pi \left(x_{1} \right)\right)=\nu '_{1} \left(\pi ^{-1 } \left(x_{2} \right)\right)\oplus \nu '_{2} \left(\pi \left(x_{1} \right)\right)=} \\ {\, \, \, \, \, \, \, \, \, =\nu '_{2} \circ \pi \left(x_{1} \right)\oplus \nu '_{1} \circ \pi ^ {-1} \left(x_{2} \right)=\mu '_{1} \left(x_{1} \right)\oplus \mu '_{2} \left(x_{2} \right)=\mu '\left(x_{1} \oplus x_{2} \right)} \end{array}\] we obtain that \[\frac{d\mu '}{d\nu '\circ \alpha ^{- 1} } =1,\] i.e. $Aut\, Z_{M} \left(\mu ';\nu '\right)\ne \emptyset $. From this it follows from theorem \ref{th2} that the log-algebras $L_{\log } \left(Z_{M} ;\mu '\right)$ and $L_{\log } \left(Z_{M} ; \nu '\right)$ *-isomorphic.

\noindent Now let $\mu _{2} $ and $\nu _{2} $ be the extensions of the exact normal semifinite traces $\mu '_{2} $ and $\nu '_{2} $ from $Z_{2} $ to $M^{II} $, respectively. Then $\mu =\mu _{1} \oplus \mu _{2} $ and $\nu =\nu _{1} \oplus \nu _{2} $ are exact normal semifinite traces on $M$. From the equality $Aut\, M^{I} \left(\mu _{1} ;\nu _{1} \right)=\emptyset $ it follows that
\[AutM\left(\mu ;\nu \right)\subset Aut\, M\left(\mu _{1} \oplus 0;\nu _{1} \oplus 0\right)\ne \emptyset \]
where 0 is the zero trace on $M^{II} $. In the last inclusion, in particular, the fact is used that any *-automorphism of the algebra $M$ maps $M^{I} $ to $M^{I} $. Thus, we have obtained that the log-algebras $L_{\log } \left(M;\mu \right)$ and $L_{\log } \left(M;\nu \right)$ are not *-isomorphic, while the corresponding log-algebras on the center are *-isomorphic.

From Theorem \ref{th2} and the equality $\frac{d\mu }{d\nu } =\frac{d\mu '}{d\nu '}$
the following theorem we obtain

\begin{theorem}\label{th4} The *-isomorphism of the log-algebras $L_{\log } \left(Z_{M} ;\mu '\right)$ and $L_{\log } \left(Z_{M} ;\nu '\right)$ implies the *-isomorphism of the log-algebras $L_{\log } \left(M;\mu \right)$ and $L_{\log } \left(M;\nu \right)$ if and only if at least one *-automorphism of $Aut\, Z_{M} \left(\mu ';\nu '\right)$ can be extended to an *-automorphism on $M$.
\end{theorem}
\sloppy

\begin{proof} If log-algebras $L_{\log } \left(M;\mu \right)$ and $L_{\log } \left(M;\nu \right)$ are *-isomorphic, then by Theorem \ref{th3} there exists a *-automorphism $\alpha \in AutM\left(\mu ;\nu \right)$.
The restriction $\alpha _{0} $ of this *-automorphism to $Z_{M} $
satisfies the condition $\alpha _{0} \in Aut(Z_{M} \left(\mu ';\nu '\right))$, since
\[\frac{d\mu }{d\nu \circ \alpha ^{-1} } =\frac{d\mu '}{d\nu '\circ \alpha ^{-1} } .\]

Therefore, the constructed *-automorphism $\alpha _{0} $ from $Aut\, Z_{M} \left(\mu ';\nu '\right)$ extends to an *-automorphism on $M$.

We now prove the converse implication. Let the *-automorphism \[\alpha _{0} \in Aut\, Z_{M} \left(\mu ';\nu '\right)\]
extend to an *-automorphism $\alpha $ on $M$. That this is not always the case follows from Example 1, i.e., not every *-automorphism on $Z_{M} $ extends to an *-automorphism on $M$. Since \[\frac{d\mu }{d\nu \circ \alpha ^{-1} } =\frac{d\mu '}{d\nu '\circ \alpha _{0} {}^{-1} } \in L_{0} \left(Z_{M} ;\nu '\circ \alpha _{0}^{-1} \right)\subset L_{0} \left(M;\nu \circ \alpha ^{-1} \right),\] then we obtain that \[\alpha \in AutM\left(\mu ;\nu \right),\] i.e. $L_{\log } \left(M;\mu \right)$ and $L_{\log } (M,\nu )$ are isomorphic.
\end{proof}

Let $X$ be an arbitrary complete Boolean algebra, $e\in X,$ $X_{e} =\{ g\in X:g\le e\} $. By $\tau (X_{e} )$, we denote the minimal cardinality of a set dense in $X_{e} $. An infinite complete Boolean algebra $X$ is called homogeneous if $\tau (X_{e} )=\tau (X_{g} )$ for any nonzero $e,g\in X.$ The cardinality $\tau (X)$ is called the weight of the homogeneous Boolean algebra $X$.

Let $\nabla $ be a complete nonatomic Boolean algebra and $\mu $ be a strictly positive countably additive $\sigma $-finite measure on $\nabla $. Then, the decomposition of $\nabla $ into homogeneous components is at most countable.

\begin{definition} Denote by $\{ \nabla _{s_{i} } \} $ the homogeneous components of the Boolean algebra $\nabla $ for which
\[\tau _{s_{i} } =\tau (\nabla _{s_{i} } )<\tau _{s_{i+1} } ,\; \mu (s_{i} )=\infty ,\]
and by $\{ \nabla _{u_{i} } \} $ we denote the homogeneous components of the Boolean algebra $\nabla $ for which $\tau _{u_{i} } =\tau (\nabla _{u_{i} } )<\tau _{u_{i+1} } ,\; \mu (u_{i} )<\infty .$ Here $s_{i} $ and $u_{i} $ are the units of the Boolean algebras $\nabla _{s_{i} } $ and $\nabla _{u_{i} } ,$ respectively. Then the matrix \[\left(\begin{array}{ccc} {\tau _{s_{1} } } & {\tau _{s_{2} } } & {\ldots } \\ {\tau _{u_{1} } } & {\tau _{u_{2} } } & {\ldots } \\ {\mu _{1} } & {\mu _{2} } & {\ldots } \end{array}\right),\] is uniquely defined, which we will call the \textit{passport} of the Boolean algebra $\nabla $ with $\sigma $-finite measure $\mu $.
\end{definition}

In particular, the concept of a two-line passport for finite measures was introduced in \cite{9}.

The Theorem 6 in \cite{2} establishes a necessary and sufficient condition for isomorphism of commutative $\log $-algebras.

\begin{theorem}\label{th5}
Let $\mu $ and $\nu $ be strictly positive $\sigma $-finite measures on non-atomic complete Boolean algebras $X$ and $Y$, respectively. Let $\left(\begin{array}{ccc} {\tau _{s_{1} } } & {\tau _{s_{2} } } & {\ldots } \\ {\tau _{u_{1} } } & {\tau _{u_{2} } } & {\ldots } \\ {\mu _{1} } & {\mu _{2} } & {\ldots } \end{array}\right)$  and  $\left(\begin{array}{ccc} {\tau _{s_{1} } } & {\tau _{s_{2} } } & {\ldots } \\ {\tau _{u_{1} } } & {\tau _{u_{2} } } & {\ldots } \\ {\mu _{1} } & {\mu _{2} } & {\ldots } \end{array}\right)$ are passports of Boolean algebras $(X,\mu )$ and $(Y,\nu ).$ Then the following conditions are equivalent:

$(i)$ $L_{\log } \left(X,\mu \right)$ and $L_{\log } \left(Y,\nu \right)$ are *-isomorphic;

$(ii)$ a) the first and second lines of the passports of Boolean algebras $\left(X,\mu \right)$ and $\left(Y,\nu \right)$ coincide;

b) the sequences $\mu _{n} \nu _{n}^{-1} $ and $\nu _{n} \mu _{n}^{-1} $ are bounded, where $\mu _{n} $ and $\nu _{n} $ are numbers from the third line of the passports of the Boolean algebras $\left(X,\mu \right)$ and $\left(Y,\nu \right)$, respectively.
\end{theorem}

\section{Criterion of isomorphism of $\log$-algebras constructed on von Neumann algebras of type $I_{n}$}

\begin{definition} A von Neumann algebra of type $I_{n} $ is a von Neumann algebra $M_{X} $ with the representation $M_{X} =L_{\infty } (X)\otimes B(H_{n} )$, where $B(H_{n} )$ is the algebra of $n\times n$ matrices \cite{10}. The center of $M_{X} $ is identified with $L_{\infty } (X)$. \end{definition}
For these von Neumann algebras, the $\log $-algebra will have the form
\[L_{\log } (M,\varphi )=L_{\log } (X,\mu )\otimes B(H_{n} ).\]
Let $M_{X} $ and $M_{Y} $ be von Neumann algebras of type $I_{n} $ with faithful normal semifinite traces $\varphi $ and $\eta $, respectively. Let $\mu $ and $\nu $ denote the restrictions of the traces of $\varphi $ and $\eta $ to the corresponding centers
\[L_{\log } (M,\varphi )=L_{\log } (X,\mu )\otimes B(H_{n} ),\]
\[L_{\log } (N,\eta )=L_{\log } (Y,\nu )\otimes B(H_{n} ).\]

From \cite[Theorem 2.3.3, Corollary 2.3.4, p. 89]{10} the following assertion follows.

\begin{proposition}\label{pr2}
Let $M_{X} $ and $M_{Y} $ be von Neumann algebras of type $I_{n} $. Then any isomorphism of their centers $Z(M_{X} )$ and $Z(M_{Y} )$ extends to an isomorphism from $M_{X} $ onto $M_{Y} $.

\end{proposition}

\begin{theorem}
The following conditions are equivalent:

$(i)$ The log algebras $L_{\log } (M_{X} ,\varphi )$ and $L_{\log } (M_{Y} ,\eta )$ are isomorphic;

The $(ii)$ Log-algebras $L_{\log } (X,\mu )$ and $L_{\log } (Y,\nu )$ are isomorphic;

$(iii)$ a) the first and second lines of the passports of the Boolean algebras $\left(X,\mu \right)$ and $\left(Y,\nu \right)$ coincide;

b) the sequences $\mu_{n} \nu_{n}^{-1} $ and $\nu _{n} \mu _{n}^{-1} $ are bounded, where $\mu_{n} $ and $\nu_{n} $ are numbers from the third line of the passports of the Boolean algebras $\left(X,\mu \right)$ and $\left(Y,\nu \right)$, respectively.
\end{theorem}

\begin{proof}
The equivalence of $(i)\Leftrightarrow (ii)$ follows from Theorem \ref{th4} and Proposition \ref{pr2}. The equivalence of $(ii)\Leftrightarrow (iii)$ follows from Theorem \ref{th5}.
\end{proof}

\end{document}